\documentclass[english]{amsart}
\usepackage[T1]{fontenc}
\usepackage[latin9]{inputenc}
\usepackage{amsthm}
\usepackage{amssymb}

\makeatletter
\theoremstyle{plain}
\newtheorem{thm}{Theorem}
  \theoremstyle{definition}
  \newtheorem{defn}[thm]{Definition}
  \theoremstyle{remark}
  \newtheorem{rem}[thm]{Remark}
  \theoremstyle{plain}
  \newtheorem{lem}[thm]{Lemma}
  \theoremstyle{plain}
  \newtheorem{prop}[thm]{Proposition}
 \theoremstyle{definition}
  \newtheorem{example}[thm]{Example}

\makeatother

\makeatother

\usepackage{babel}

\begin{document}

\title{Morita Transforms of Tensor Algebras}

\author{Paul S. Muhly and Baruch Solel}

\address{Department of Mathematics\\
 University of Iowa\\
 Iowa City, IA 52242}

\email{pmuhly@math.uiowa.edu}

\address{Department of Mathematics\\
 Technion\\
 32000 Haifa\\
 Israel}

\email{mabaruch@techunix.technion.ac.il}


\thanks{The research of both authors was supported in part by a U.S.-Israel
Binational Science Foundation grant. The second author gratefully acknowledges
support from the Lowengart Research Fund.}

\begin{abstract}
We show that if $M$ and $N$ are $C^{*}$-algebras and if $E$ (resp.
$F$) is a $C^{*}$-correspondence over $M$ (resp. $N$), then a
Morita equivalence between $(E,M)$ and $(F,N)$ implements a isometric
functor between the categories of Hilbert modules over the tensor
algebras of $\mathcal{T}_{+}(E)$ and $\mathcal{T}_{+}(F)$. We show
that this functor maps absolutely continuous Hilbert modules to absolutely
continuous Hilbert modules and provides a new interpretation of Popescu's
reconstruction operator. 
\end{abstract}
\maketitle

\section{Introduction}

Suppose $M$ is a $C^{*}$-algebra and $E$ is a $C^{*}$-correspondence
over $M$ in the sense of \cite{MS98b}. This means, first of all,
that $E$ is a (right) Hilbert $C^{*}$-module, and secondly, that
if $\mathcal{L}(E)$ denotes the space of all bounded adjointable
module maps on $E$, then $E$ becomes a left $M$-module via a $C^{*}$-homomorphism
$\varphi_{M}$ from $M$ into $\mathcal{L}(E)$. To emphasize the
connection between $E$ and $M$, we will call the pair, $(E,M)$,
a \emph{$C^{*}$-correspondence pair}. Form the \emph{Fock space}
built from $(E,M$), $\mathcal{F}(E)$. This is the direct sum $\sum_{n\geq0}E^{\otimes n}$,
where $E^{\otimes n}$ is the internal tensor product of $E$ with
itself $n$ times. (The tensor products are balanced over $M$.) The
Fock space $\mathcal{F}(E)$ is, itself, a $C^{*}$-correspondence
over $M$ and we write $\varphi_{M\infty}$ for the left action of
$M$. For $\xi\in E$, $T_{\xi}$ denotes the creation operator on
$\mathcal{F}(E)$ determined by $\xi$, i.e., for $\eta\in\mathcal{F}(E)$,
$T_{\xi}\eta=\xi\otimes\eta$. We let $\mathcal{T}_{+}(E)$ denote
the norm closed subalgebra of $\mathcal{L}(\mathcal{F}(E))$ generated
by $\varphi_{M\infty}(M)$ and $\{T_{\xi}\mid\xi\in E\}$, and we
call $\mathcal{T}_{+}(E)$ \emph{the tensor algebra of} $E$ or of
$(E,M)$. In \cite[Defintion 2.1]{MS2000}, we introduced the following
definition.
\begin{defn}
\label{def:Morita-equi-Wstar-correspondence}We say that two $C^{*}$-correspondence
pairs $(E,M)$ and $(F,N)$ are \emph{Morita equivalent} in case there
is a $C^{*}$-equivalence bimodule $\mathcal{X}$ in the sense of
\cite[Definition 7.5]{mR74a} such that \[
\mathcal{X}\otimes_{N}F\simeq E\otimes_{M}\mathcal{X}\]
 as $C^{*}$-correspondences. In this case, we say that $\mathcal{X}$
\emph{implements} a Morita equivalence between $(E,M)$ and $(F,N)$. 
\end{defn}
Observe that the equation $\mathcal{X}\otimes_{N}F\simeq E\otimes_{M}\mathcal{X}$
is equivalent to the equation $\mathcal{X}\otimes_{N}F\otimes_{N}\widetilde{\mathcal{X}}\simeq E$
and to the equation $F\simeq\widetilde{\mathcal{X}}\otimes_{M}E\otimes_{M}\mathcal{X}$,
where $\widetilde{\mathcal{X}}$ is the dual or opposite module of
$\mathcal{X}$. We showed there that if $(E,M)$ and $(F,N)$ are
Morita equivalent, then the tensor algebras $\mathcal{T}_{+}(E)$
and $\mathcal{T}_{+}(F)$ are Morita equivalent in the sense of \cite{BMP2000}.
It follows that $\mathcal{T}_{+}(E)$ and $\mathcal{T}_{+}(F)$ have
isometrically isomorphic representation theories. However, when looking
at the formulas involved in the isomorphism between the representation
theories, certain details become obscure. Our objective in this note
is to show that simply tensoring with $\mathcal{X}$ implements an
explicit isometric isomorphism between the representation theories
of $\mathcal{T}_{+}(E)$ and $\mathcal{T}_{+}(F)$ in a fashion that
preserves important properties that we shall introduce shortly. The
first step is to have a clear picture of the representation theory
of an operator tensor algebra.

\section{The Representations of $\mathcal{T}_{+}(E)$}

We begin with a restatement of Theorem 3.10 in \cite{MS98b}. 
\begin{thm}
\label{thm:Disintigration}Let $\rho$ be a completely contractive
representation of $\mathcal{T}_{+}(E)$ on a Hilbert space $H$. Define
$\sigma:M\to B(H)$ by the formula $\sigma(a)=\rho\circ\varphi_{M\infty}(a)$
and define $T:E\to B(H)$ by the formula $T(\xi)=\rho(T_{\xi})$.
Then $\sigma$ is a $C^{*}$-representation of $M$ on $H$ and $T$
is a completely contractive bimodule map in the sense that $T(\varphi_{M}(a)\xi b)=\sigma(a)T(\xi)\sigma(b)$
for all $a,b\in M$ and all $\xi\in E$. Conversely, given a $C^{*}$-representation
$\sigma:M\to B(H)$ and a completely contractive bimodule map $T:E\to B(H)$,
there is a unique completely contractive representation $\rho:\mathcal{T}_{+}(E)\to B(H)$
such that $\sigma=\rho\circ\varphi_{M}$ and $T(\xi)=\rho(T_{\xi})$
for all $\xi\in E$. 
\end{thm}
If $T$ is a completely contractive bimodule map with respect to a
$C^{*}$-representation $\sigma$ of $M$, then we call $(T,\sigma)$
a \emph{completely contractive covariant pair}. We call the completely
contractive representation $\rho$ of $\mathcal{T}_{+}(E)$ that $(T,\sigma)$
determines the \emph{integrated form} of $(T,\sigma)$ and write $\rho=T\times\sigma$.
Theorem \ref{thm:Disintigration} begs the question: How does one
construct completely contractive covariant pairs? For this purpose,
we need to recall the definition of Rieffel's induced representation
\cite{mR74b}. If $\sigma:M\to B(H)$ is a Hilbert space representation
of $M$, then we may build the Hilbert space $E\otimes_{\sigma}H$,
which is the separated completion of the algebraic tensor product
$E\otimes H$ in the pre-inner product defined by the formula \[
\langle\xi\otimes h,\eta\otimes k\rangle:=\langle h,\sigma(\langle\xi,\eta\rangle)k\rangle,\qquad\xi\otimes h,\eta\otimes k\in E\otimes H.\]
 The representation $\sigma^{E}$ of $\mathcal{L}(E)$ on $E\otimes_{\sigma}H$
defined by the formula, $\sigma^{E}(T):=T\otimes I$, $T\in\mathcal{L}(E)$,
is called the representation of $\mathcal{L}(E)$ \emph{induced} by
$\sigma$. The following theorem is essentially Lemma 2.5 of \cite{MSHardy}.
\begin{thm}
\label{thm:Intertwiner-Covariant}Let $\sigma:M\to B(H)$ be a $C^{*}$-representation.
A completely contractive linear map $T$ from $E$ to $B(H)$ is a
bimodule map with respect to $\sigma$ if and only if there is an
operator $\widetilde{T}:E\otimes_{\sigma}H\to H$ with $\Vert\widetilde{T}\Vert\leq1$
such that $\widetilde{T}\sigma^{E}\circ\varphi_{M}=\sigma\widetilde{T}$
and $T(\xi)h=\widetilde{T}(\xi\otimes h)$, for all $\xi\otimes h\in E\otimes_{\sigma}H$. 
\end{thm}
Thus the completely contractive bimodule maps are in bijective correspondence
with (contractive) intertwiners. The space of intertwiners of $\sigma$
and $\sigma^{E}\circ\varphi_{M}$ is a key player in our theory and
to keep the notation manageable, when there is no risk of confusion
in the context under discussion, we shall not distinguish notationally
between bimodule maps $T$ and the corresponding intertwiner $\widetilde{T}$.
Further, for reasons that will be explained in a minute, we frequently
also denote bimodule maps by lower case fraktur letters from the end
of the alphabet, as we do now. 
\begin{defn}
\label{def:sigma-dual}Let $\sigma:M\to B(H)$ be a $C^{*}$-representation.
\emph{The $\sigma$-dual} of $E$, denoted $E^{\sigma}$, is defined
to be $\{\mathfrak{z}\in B(H,E\otimes_{\sigma}H)\mid\mathfrak{z}\sigma=\sigma^{E}\circ\varphi_{M}\mathfrak{z}\}$.
We write $E^{\sigma*}$ for the space $\{\mathfrak{z}^{*}\mid\mathfrak{z}\in E^{\sigma}\}$
and we write $\mathbb{D}(E^{\sigma*})$ for $\{\mathfrak{z}^{*}\in E^{\sigma*}\mid\Vert\mathfrak{z}^{*}\Vert<1\}$,
i.e., $\mathbb{D}(E^{\sigma*})$ is the open unit ball in $E^{\sigma*}$. 
\end{defn}
Thanks to Theorem \ref{thm:Intertwiner-Covariant}, $\overline{\mathbb{D}(E^{\sigma*})}$
labels \emph{all} the completely contractive representations $\rho$
of $\mathcal{T}_{+}(E)$ with the property that $\rho\circ\varphi_{M\infty}=\sigma$.
The reason we introduced $E^{\sigma}$, instead of focusing exclusively
on $E^{\sigma*}$ is that $E^{\sigma}$ is a $W^{*}$-correspondence
over $\sigma(M)'$. (A $W^{*}$-correspondence is a $C^{*}$-correspondence
with some additional structure that we discuss below.) For $\mathfrak{z}_{1},\mathfrak{z}_{2}\in E^{\sigma}$,
$\langle\mathfrak{z}_{1},\mathfrak{z}_{2}\rangle:=\mathfrak{z}_{1}^{*}\mathfrak{z}_{2}$,
and the $\sigma(M)'$-bimodule actions are given by the formula\[
a\cdot\mathfrak{z}\cdot b:=(I_{E}\otimes a)\mathfrak{z}b,\qquad a,b\in\sigma(M)',\mathfrak{z}\in E^{\sigma},\]
where the products on the right are just composition of the maps involved.
The reason for introducing the notation $\overline{\mathbb{D}(E^{\sigma*})}$
and writing elements in this ball as lower case $\mathfrak{z}$'s,
$\mathfrak{w}$'s, etc., is that we may view an element $F\in\mathcal{T}_{+}(E)$
as a function $\widehat{F}$ on $\overline{\mathbb{D}(E^{\sigma*})}$
via the formula \[
\widehat{F}(\mathfrak{z}):=\mathfrak{z}\times\sigma(F),\qquad\mathfrak{z}\in\overline{\mathbb{D}(E^{\sigma*})}.\]
Functions of the form $\widehat{F}$ are bonafide $B(H_{\sigma})$-valued
analytic functions on $\mathbb{D}(E^{\sigma*})$ with additional very
interesting properties, and they can be studied with function-theoretic
techniques. (See \cite{MSHardy,MS08,MS09}.) For the purpose of emphasizing
the function-theoretic properties of the $\widehat{F}$'s, it seems
preferable to write their arguments as $\mathfrak{z}$'s instead of
$T$'s. But when representation-theoretic features need emphasis,
the use of $T$ and $T\times\sigma$ is sometimes preferable.

\section{The Functor}

Our objective in this section is to show that Morita equivalence of
$C^{*}$-correspondence pairs $(E,M)$ and $(F,N)$ gives rise to
a natural isometric isomorphism between representation theory of $\mathcal{T}_{+}(E)$
and $\mathcal{T}_{+}(F)$. 
\begin{thm}
\label{thm:functor}Suppose $(E,M)$ and $(F,N)$ are Morita equivalent
$C^{*}$-correspondence pairs via an $M,N$-equivalence bimodule $\mathcal{X}$
and correspondence isomorphism $W:E\otimes_{M}\mathcal{X}\to\mathcal{X}\otimes_{N}F$.
Suppose further that $\sigma:N\to B(H)$ is a $C^{*}$-representation
and let $\sigma^{\mathcal{X}}:M\to B(\mathcal{X}\otimes_{\sigma}H)$
be the representation of $M$ induced by $\mathcal{X}$. Then for
each $\mathfrak{z}^{*}\in\overline{\mathbb{D}(F^{\sigma*})}$, $\mathfrak{z}^{*\mathcal{X}}:=(I_{\mathcal{X}}\otimes\mathfrak{z}^{*})(W\otimes I_{H})$
lies in $\overline{\mathbb{D}(E^{\sigma^{\mathcal{X}}*})}$ and the
map $\mathfrak{z}^{*}\to\mathfrak{z}^{*\mathcal{X}}$ is an isometric
surjection onto $\overline{\mathbb{D}(E^{\sigma^{\mathcal{X}}*})}$.\end{thm}
\begin{proof}
For $\mathfrak{z}^{*}\in\overline{\mathbb{D}(F^{\sigma*})}$ set $\mathfrak{z}_{1}^{*}:=\left[\begin{array}{cc}
0 & \mathfrak{z}^{*}\\
0 & 0\end{array}\right]$ acting on $H\oplus(F\otimes_{\sigma}H)$. Then $\mathfrak{z}_{1}^{*}$
commutes with $\left[\begin{array}{cc}
\sigma & 0\\
0 & \sigma^{F}\circ\varphi_{N}\end{array}\right]$. Consequently, $I_{\mathcal{X}}\otimes\mathfrak{z}_{1}^{*}=\left[\begin{array}{cc}
0 & I_{\mathcal{X}}\otimes\mathfrak{z}^{*}\\
0 & 0\end{array}\right]$ acting on $\mathcal{X}\otimes_{\sigma}H\oplus\mathcal{X}\otimes_{\sigma^{F}\circ\varphi_{N}}(F\otimes_{\sigma}H)$
commutes with $\left[\begin{array}{cc}
\sigma^{\mathcal{X}} & 0\\
0 & (\sigma^{F}\circ\varphi_{N})^{\mathcal{X}}\end{array}\right]$. Since $W\otimes I_{H}:E\otimes\mathcal{X}\otimes_{\sigma}H\to\mathcal{X}\otimes F\otimes_{\sigma}H$
intertwines $(\sigma^{\mathcal{X}})^{E}\circ\varphi_{M}$ and $(\sigma^{F}\circ\varphi_{N})^{\mathcal{X}}$
by hypothesis, we see that $\left[\begin{array}{cc}
0 & (I_{\mathcal{X}}\otimes\mathfrak{z}^{*})(W\otimes I_{H})\\
0 & 0\end{array}\right]$ commutes with $\left[\begin{array}{cc}
\sigma^{\mathcal{X}} & 0\\
0 & (\sigma^{\mathcal{X}})^{E}\circ\varphi_{M}\end{array}\right]$. Since $\Vert(I_{\mathcal{X}}\otimes\mathfrak{z}^{*})(W\otimes I_{H})\Vert=\Vert\mathfrak{z}^{*}\Vert$,
it follows that $\mathfrak{z}^{*\mathcal{X}}:=(I_{\mathcal{X}}\otimes\mathfrak{z}^{*})(W\otimes I_{H})$
lies in $\overline{\mathbb{D}(E^{\sigma^{\mathcal{X}}*})}$ and that
the map $\mathfrak{z}^{*}\to\mathfrak{z}^{*\mathcal{X}}$ is isometric.
Finally, to see that the map is surjective, we appeal to \cite[Theorem 6.23]{mR74b}:
Let $\mathfrak{w}^{*}\in\overline{\mathbb{D}(E^{\sigma^{\mathcal{X}}*})}$.
Then $\mathfrak{w}^{*}$ intertwines $\sigma^{\mathcal{X}}$ and $(\sigma^{\mathcal{X}})^{E}\circ\varphi_{M}$
by hypothesis. Consequently, $\mathfrak{w}^{*}(W\otimes I_{H})^{-1}$
intertwines $\sigma^{\mathcal{X}}$ and $(\sigma^{F}\circ\varphi_{N})^{\mathcal{X}}$.
That is $\left[\begin{array}{cc}
0 & \mathfrak{w}^{*}(W\otimes I_{H})^{-1}\\
0 & 0\end{array}\right]$ lies in the commutant of $\left[\begin{array}{cc}
\sigma^{\mathcal{X}}(\mathcal{L}(\mathcal{X})) & 0\\
0 & (\sigma^{F}\circ\varphi_{N})^{\mathcal{X}}(\mathcal{L}(\mathcal{X}))\end{array}\right]=\left[\begin{array}{cc}
\sigma & 0\\
0 & (\sigma^{F}\circ\varphi_{N})\end{array}\right]^{\mathcal{X}}(\mathcal{L}(\mathcal{X}))$ and so, by \cite[Theorem 6.23]{mR74b}, $\left[\begin{array}{cc}
0 & \mathfrak{w}^{*}(W\otimes I_{H})^{-1}\\
0 & 0\end{array}\right]$ must have the form $I_{\mathcal{X}}\otimes\left[\begin{array}{cc}
\mathfrak{z}_{11} & \mathfrak{z}_{12}\\
\mathfrak{z}_{21} & \mathfrak{z}_{22}\end{array}\right]$, where $\left[\begin{array}{cc}
\mathfrak{z}_{11} & \mathfrak{z}_{12}\\
\mathfrak{z}_{21} & \mathfrak{z}_{22}\end{array}\right]$ lies in the commutant of $\left[\begin{array}{cc}
\sigma & 0\\
0 & (\sigma^{F}\circ\varphi_{N})\end{array}\right]$. Since $I_{\mathcal{X}}\otimes\left[\begin{array}{cc}
\mathfrak{z}_{11} & \mathfrak{z}_{12}\\
\mathfrak{z}_{21} & \mathfrak{z}_{22}\end{array}\right]$ maps
$\mathcal{X}\otimes_{N}(F\otimes_{\sigma}H)$ to $\mathcal{X}\otimes_{\sigma}H$
and is zero on $\mathcal{X}\otimes_{\sigma}H$, it follows that $\left[\begin{array}{cc}
\mathfrak{z}_{11} & \mathfrak{z}_{12}\\
\mathfrak{z}_{21} & \mathfrak{z}_{22}\end{array}\right]=\left[\begin{array}{cc}
0 & \mathfrak{z}_{12}\\
0 & 0\end{array}\right]$ for $\mathfrak{z}_{12}\in\overline{\mathbb{D}(F^{\sigma*})}$, proving
that the map $\mathfrak{z}^{*}\to\mathfrak{z}^{*\mathcal{X}}$ is
surjective. \end{proof}
\begin{defn}
If $(E,M)$ and $(F,N)$ are Morita equivalent $C^{*}$-correspondence
pairs via an equivalence $M,N$-bimodule $\mathcal{X}$, then the
map $(T,\sigma)\to(T^{\mathcal{X}},\sigma^{\mathcal{X}})$ from the
representation theory of $\mathcal{T}_{+}(E)$ to the representation
theory of $\mathcal{T}_{+}(F)$ defined by $\mathcal{X}$ will be
called the \emph{Morita transform} determined by $\mathcal{X}$. 
\end{defn}
We like to think of the Morita transform as a generalized conformal
map.

\section{Morita Equivalence and Absolute Continuity}

Our focus in this section will be on Morita equivalence in the context
of $W^{*}$-algebras and $W^{*}$-correspondences. As we noted above,
a $W^{*}$-correspondence is a $C^{*}$-correspondence with additional
structure. We begin by highlighting what the additional structure
is and how to deal with it. So, throughout this section $M$ and $N$
will be $W^{*}$-algebras and $E$ (resp. $F$) will be a $W^{*}$-correspondence
over $M$ (resp. $N$). This means, in particular, that $E$ and $F$
are \emph{self-dual} Hilbert $C^{*}$-modules over $M$ and $N$,
respectively, in the sense of Paschke \cite[Section 3, p. 449]{wP73},
and that the left actions of $M$ and $N$ are given by \emph{normal}
representations, $\varphi_{M}$ and $\varphi_{N}$ of $M$ and $N$
into $\mathcal{L}(E)$ and $\mathcal{L}(F)$, respectively. (Recall
that Paschke showed that in the setting of self-dual Hilbert modules
over $W^{*}$-algebras, every continuous module map is adjointable
and $\mathcal{L}(E)$ is a $W^{*}$-algebra by \cite[Corollary 3.5 and Proposition 3.10]{wP73}.)
To avoid technical distractions, we assume that $\varphi_{M}$ and
$\varphi_{N}$ are faithful and unital. 

A key role in this theory is played by Paschke's Theorem 3.2 in \cite{wP73},
which says among other things that any Hilbert $C^{*}$-module $E$
over a $W^{*}$-algebra has a canonical embedding into a self-dual
Hilbert module over the algebra, which should be viewed as a canonical
completion of $E$. This allows us to perform $C^{*}$-algebraic constructions
and pass immediately to the completions to obtain $W^{*}$-algebraic
objects. For instance, if $E$ is a Hilbert $W^{*}$-module over $M$,
then we may form the $C^{*}$-tensor square, $E^{\otimes2}=E\otimes_{M}E$,
which is not, in general, a $W^{*}$-correspondence over $M$. However,
its self-dual completion is. More generally, we can form the \emph{$C^{*}$-Fock
space} built from $(E,M$), $\mathcal{F}_{c}(E)$, as we did at the
outset of this note. Then we let $\mathcal{F}(E)$ be the self-dual
completion of $\mathcal{F}_{c}(E)$ in the sense of \cite[Theorem 3.2]{wP73},
and call $\mathcal{F}(E)$ \emph{the Fock space of the $W^{*}$-correspondence
$E$.} Similarly, we form $\mathcal{F}_{c}(F)$ and $\mathcal{F}(F)$.
We write $\varphi_{M\infty}$ for the left action of $M$ on both
$\mathcal{F}_{c}(E)$ and on $\mathcal{F}(E)$. This should cause
no confusion, since every element of $\mathcal{L}(\mathcal{F}_{c}(E))$
has a unique extension to an element of $\mathcal{L}(\mathcal{F}(E))$,
by \cite[Corollary 3.7]{wP73}, and the process of mapping each element
in $\mathcal{L}(\mathcal{F}_{c}(E))$ to its extension in $\mathcal{L}(\mathcal{F}(E))$
gives an isometric embedding of $\mathcal{L}(\mathcal{F}_{c}(E))$
in $\mathcal{L}(\mathcal{F}(E))$. Likewise, $\varphi_{N\infty}$
denotes the left action of $N$ on both $\mathcal{F}_{c}(F)$ and
$\mathcal{F}(F)$. The creation operator $T_{\xi}$ on $\mathcal{F}_{c}(E)$
determined by $\xi\in E$, therefore has a unique extension to $\mathcal{F}(E)$
and we do not distinguish notationally between the original and the
extension. But in the $W^{*}$-setting we let $\mathcal{T}_{+}(E)$
denote the norm closed subalgebra of $\mathcal{L}(\mathcal{F}(E))$
generated by $\varphi_{M\infty}(M)$ and $\{T_{\xi}\mid\xi\in E\}$,
and we call $\mathcal{T}_{+}(E)$ \emph{the tensor algebra of} $E$
or of $(E,M)$. That is, we focus on the tensor algebra as living
on the $W^{*}$-Fock space $\mathcal{F}(E)$. We view $\mathcal{T}_{+}(F)$
similarly. Finally, we let $H^{\infty}(E)$ denote the ultra-weak
closure of $\mathcal{T}_{+}(E)$ in $\mathcal{L}(\mathcal{F}(E))$,
and we let $H^{\infty}(F)$ denote the ultra-weak closure of $\mathcal{T}_{+}(F)$
in $\mathcal{L}(\mathcal{F}(F))$. The algebras $H^{\infty}(E)$ and
$H^{\infty}(F)$ are called the \emph{Hardy algebras} of $E$ and
$F$, respectively. 

In the special case when $M=\mathbb{C}=E$, we see that $\mathcal{F}_{c}(E)=\mathcal{F}(E)=\ell^{2}(\mathbb{N})$,
$\mathcal{T}_{+}(E)$ is the disc algebra $A(\mathbb{D})$ and $H^{\infty}(E)=H^{\infty}(\mathbb{T})$.
More generally, when $M=\mathbb{C}$ and $E=\mathbb{C}^{d}$, $\mathcal{T}_{+}(E)$
is Popescu's noncommutative disc algebra and $H^{\infty}(E)$ is his
noncommutative Hardy algebra \cite{gP91}. Somewhat later, Davidson and Pitts 
studied $H^{\infty}(\mathbb{C}^d)$ under the name
\emph{noncommutative analytic Toeplitz algebra} \cite{DP98}.

\begin{defn}
\label{def:Morita-equi-Wstar-correspondence-1}If $M$ and $N$ are
$W^{*}$-algebras and if $E$ and $F$ are $W^{*}$-correspondences
over $M$ and $N$, respectively, we say that $(E,M)$ and $(F,N)$
are \emph{Morita equivalent} in case there is a self-dual $M-N$ equivalence
bimodule $\mathcal{X}$ in the sense of \cite[Definition 7.5]{mR74a}
such that \[
\mathcal{X}\otimes_{N}F\simeq E\otimes_{M}\mathcal{X}\]
as $W^{*}$-correspondences. In this case, we say that $\mathcal{X}$
\emph{implements} a Morita equivalence between $(E,M)$ and $(F,N)$. 
\end{defn}
We emphasize that the modules $\mathcal{X}\otimes_{N}F$ and $E\otimes_{M}\mathcal{X}$
are self-dual completions of the balanced $C^{*}$-tensor products.
A completely contractive representation of a $W^{*}$-correspondence
pair $(E,M)$ on a Hilbert space $H$ is a pair $(T,\sigma)$ where
$\sigma$ is a normal representation of $M$ on $H$ and where $T$
is an ultra-weakly continuous, completely contractive bimodule map
from $E$ to $B(H)$. However, as we noted in \cite[Remark 2.6]{MSHardy},
the ultra-weak continuity of $T$ follows automatically from the bimodule
property of $T$ and the normality of $\sigma$.

Our goal is to show that Morita equivalence in the sense of Definition
\ref{def:Morita-equi-Wstar-correspondence-1} preserves absolute continuity
in the sense of the following definition, which was inspired by the important
paper of Davidson, Li and Pitts \cite{DLP2005}. 
\begin{defn}
\label{def:absolute continuity}Let $(T,\sigma)$ be a completely
contractive covariant representation of $(E,M)$ on $H$ and assume
that $\sigma$ is a normal representation of $M$. Then a vector $x\in H$
is called \emph{absolutely continuous} if and only if the functional
$a\to\langle(T\times\sigma)(a)x,x\rangle,\,\, a\in\mathcal{T}_{+}(E)$,
extends to an ultra-weakly continuous linear functional on $H^{\infty}(E)$.
The collection of all absolutely continuous vectors in $H$ is denoted
$\mathcal{V}_{ac}(T,\sigma)$, and we say $(T,\sigma)$ and $T\times\sigma$
are absolutely continuous in case $\mathcal{V}_{ac}(T,\sigma)=H$.\end{defn}
\begin{rem}
The definition of an absolutely continuous vector just given is not quite
the one given in \cite[Definition 3.1]{MS2010}. However, by \cite[Remark 3.2]{MS2010},
it is equivalent to the one given there. Also, by virtue of \cite[Theorem 4.11]{MS2010},
$T\times\sigma$ extends to an ultra-weakly continuous completely
contractive representation of $H^{\infty}(E)$ if and only $T\times\sigma$
is absolutely continuous. \end{rem}
\begin{thm}
\label{thm:Absolute_continuity_preservation}Suppose that $(E,M)$
and $(F,N)$ are $W^{*}$-correspondence pairs that are Morita equivalent
via an equivalence bimodule $\mathcal{X}$. If $(\mathfrak{z}^{*},\sigma)$
is a completely contractive covariant representation of $(F,N)$,
where $\sigma$ is normal, then \begin{equation}
\mathcal{X}\otimes_{\sigma}\mathcal{V}_{ac}(\mathfrak{z}^{*},\sigma)=\mathcal{V}_{ac}(\mathfrak{z}^{\mathcal{X}^{*}},\sigma^{\mathcal{X}}).\label{eq:abs_cont_subspace}\end{equation}
 In particular, $(\mathfrak{z}^{*},\sigma)$ is absolutely continuous
if and only if $(\mathfrak{z}^{\mathcal{X}^{*}},\sigma^{\mathcal{X}})$
is absolutely continuous. 
\end{thm}
The proof of this theorem rests on a calculation of independent interest.
Recall that each $\mathfrak{z}\in\overline{\mathbb{D}(F^{\sigma})}$
determines a completely positive map $\Phi_{\mathfrak{z}}$ on $\sigma(N)'$
via the formula\[
\Phi_{\mathfrak{z}}(a):=\mathfrak{z}^{*}(I_{E}\otimes a)\mathfrak{z},\qquad a\in\sigma(N)'.\]
Recall, also, that the commutant of $\sigma^{\mathcal{X}}(M)$ is
$I_{\mathcal{X}}\otimes\sigma(N)'$, by \cite[Theorem 6.23]{mR74b}.
\begin{lem}
\label{lem:CP-induced}With the notation as in Theorem \ref{thm:Absolute_continuity_preservation},\[
\Phi_{\mathfrak{z}{}^{\mathcal{X}}}=I_{\mathcal{X}}\otimes\Phi_{\mathfrak{z}},\]
i.e., for all $a\in\sigma(N)'$, $\Phi_{\mathfrak{z}{}^{\mathcal{X}}}(I_{\mathcal{X}}\otimes a)=I_{\mathcal{X}}\otimes\Phi_{\mathfrak{z}}(a)$.\end{lem}
\begin{proof}
By Theorem \ref{thm:functor}, $\mathfrak{z}^{\mathcal{X}}=(W\otimes I_{H})^{*}(I_{\mathcal{X}}\otimes\mathfrak{z})$.
Consequently, for all $a\in\sigma(N)'$, \begin{align*}
\Phi_{\mathfrak{z}{}^{\mathcal{X}}}(I_{\mathcal{X}}\otimes a) & =\mathfrak{z}^{\mathcal{X}^{*}}(I_{E}\otimes(I_{\mathcal{X}}\otimes a))\mathfrak{z}^{\mathcal{X}}\\
= & (I_{\mathcal{X}}\otimes\mathfrak{z}^{*})(W\otimes I_{H})(I_{E\otimes\mathcal{X}}\otimes a)(W\otimes I_{H})^{*}(I_{\mathcal{X}}\otimes\mathfrak{z})\\
= & (I_{\mathcal{X}}\otimes\mathfrak{z}^{*})(I_{\mathcal{X}\otimes F}\otimes a)(I_{\mathcal{X}}\otimes\mathfrak{z})\\
= & (I_{\mathcal{X}}\otimes\mathfrak{z}^{*})(I_{\mathcal{X}}\otimes(I_{F}\otimes a))(I_{\mathcal{X}}\otimes\mathfrak{z})\\
= & I_{\mathcal{X}}\otimes\Phi_{\mathfrak{z}}(a).\end{align*}

\end{proof}
For the proof of Theorem \ref{thm:Absolute_continuity_preservation},
we need one more ingredient:
\begin{defn}
\label{def:superharmonic_operator}Let $\Phi$ be a completely positive
map on a $W^{*}$-algebra $M$. A positive element $a\in M$ is called
\emph{superharmonic} with respect to $\Phi$ in case $\Phi(a)\leq a$.
A superharmonic element $a\in M$ is called a \emph{pure} superharmonic
element in case $\Phi^{n}(a)\to0$ ultra-strongly as $n\to\infty$. \end{defn}
\begin{proof}
(of Theorem \ref{thm:Absolute_continuity_preservation}) In \cite[Theorem 4.7]{MS2010},
we proved that absolutely continuous subspace for $(\mathfrak{z}^{*},\sigma)$
is the closed linear span of the ranges of all the pure superharmonic
operators for $\Phi_{\mathfrak{z}}$, i.e., the projection onto $\mathcal{V}_{ac}(\mathfrak{z}^{*},\sigma)$
is the supremum taken over all the projections $P$, where $P$ is
the projection onto the range of a pure superharmonic operator for
$\Phi_{\mathfrak{z}}$. From Lemma \ref{lem:CP-induced} we see that
$a\in N$ is a pure superharmonic operator for $\Phi_{\mathfrak{z}}$
if and only if $I_{\mathcal{X}}\otimes a$ is a pure superharmonic
operator for $\Phi_{\mathfrak{z}{}^{\mathcal{X}}}$. Since the range
projection of $I_{\mathcal{X}}\otimes a$ is $I_{\mathcal{X}}\otimes P$,
if $P$ is the range projection of $a$, the equation \eqref{eq:abs_cont_subspace}
is immediate.
\end{proof}

\section{Stabilization and Reconstruction}

We return to the $C^{*}$-setting, although everything we will say
has an analogue in the $W^{*}$-setting. So let $N$ be a $C^{*}$-algebra
and let $F$ be a $C^{*}$-correspondence over $N$. We are out to
identify a special pair $(E,M)$ that is Morita equivalent to $(F,N)$
and is a kind of stabilization of $(F,N)$. As we will see, $(E,M)$
will have a representation theory that is closely connected to Popescu's
reconstruction operator.

Form the Fock space over $F$, $\mathcal{F}(F)$, and let $M=\mathcal{K}(\mathcal{F}(F))$.
Also, let $P_{0}$ be the projection onto the sum $F\oplus F^{\otimes2}\oplus F^{\otimes3}\oplus\cdots$
in $\mathcal{F}(F)$. Then $P_{0}$ lies in $\mathcal{L}(\mathcal{F}(F))$,
which is the multiplier algebra of $M=\mathcal{K}(\mathcal{F}(F))$.
We set $E:=P_{0}\mathcal{K}(\mathcal{F}(F))$ and endow $E$ with
its obvious structure as a right Hilbert $C^{*}$-module over $\mathcal{K}(\mathcal{F}(F))$.
Note that $\mathcal{L}(E)=P_{0}\mathcal{L}(\mathcal{F}(F))P_{0}$.
Define $R:\mathcal{F}(F)\otimes F\to\mathcal{F}(F)$ by the formula
$R(\xi\otimes f)=\xi\otimes f$, where the first $\xi\otimes f$,
the argument of $R$, is viewed as an element in $\mathcal{F}(F)\otimes_{N}F$,
while the second $\xi\otimes f$, the image of $R(\xi\otimes f)$,
is viewed as an element of $\mathcal{F}(F)$. It appears that $R$
is the identity map. However, this is only because we have suppressed
the isomorphisms between $F^{\otimes n}\otimes F$ and $F^{\otimes(n+1)}$.
The map $R$ is adjointable, and its adjoint is given by the formulae
$R^{*}(a)=0$, if $a\in N$, viewed as the zero$^{\underline{th}}$
component of $\mathcal{F}(F)$, while $R^{*}(\xi_{1}\otimes\xi_{2}\otimes\xi_{3}\otimes\cdots\otimes\xi_{n})=(\xi_{1}\otimes\xi_{2}\otimes\cdots\otimes\xi_{n-1})\otimes\xi_{n}$,
if $n\geq1$ and $\xi_{1}\otimes\xi_{2}\otimes\xi_{3}\otimes\cdots\otimes\xi_{n}$
is a decomposable element of $F^{\otimes n}\subseteq\mathcal{F}(F)$.
In particular, $RR^{*}=P_{0}$. We define $\varphi_{M}:M\to\mathcal{L}(E)$
by the formula \[
\varphi_{M}(a):=R(a\otimes I_{F})R^{*},\qquad a\in M.\]
Observe that $\varphi_{M}$ extends naturally to the multiplier algebra
of $M$, which is $\mathcal{L}(\mathcal{F}(F))$ and $\varphi_{M}(I)=P_{0}$.
Consequently, $E$ is an essential left module over $M$.
\begin{prop}
\label{pro:Stable_Morita_equivalence}If $\mathcal{X}=\mathcal{F}(F)$,
then $\mathcal{X}$ is an equivalence bimodule between $M=\mathcal{K}(\mathcal{F}(F))$
and $N$ and the map $W$ from $E\otimes_{M}\mathcal{X}$ to $\mathcal{X}\otimes_{N}F$
defined by the formula\[
W(P_{0}a\otimes\xi)=R^{*}P_{0}a\xi,\qquad P_{0}a\otimes\xi\in E\otimes_{M}\mathcal{X},\]
is an isomorphism of $M,N$-correspondences. Consequently, $(E,M)$
and $(F,N)$ are Morita equivalent. \end{prop}
\begin{proof}
By definition, $\mathcal{X}$ is an equivalence bimodule implementing
a Morita equivalence between $M$ and $N$. Also, it is clear that
$W$ is a right $N$-module map. To see that $W$ is a left $M$-module
map, it may be helpful to emphasize that the tensor product $E\otimes_{M}\mathcal{X}$
is balanced over $M$. So, if $P_{0}$ and $I$ were in $\mathcal{K}(\mathcal{F}(F))$
(which they aren't; they're only multipliers of $\mathcal{K}(\mathcal{F}(F))$),
then $P_{0}a\otimes\xi$ could be replaced by $P_{0}\otimes\xi$,
which in turn could be replaced by $I\otimes P_{0}\xi$. Further,
sending $I\otimes P_{0}\xi$ to $P_{0}\xi$ effects an isomorphism
between $E\otimes_{M}\mathcal{X}$ and $P_{0}\mathcal{F}(F)$. It
results that $W$ is effectively $R^{*}$. The following equation,
then gives the desired result.\begin{align*}
W\varphi_{M}(b)(P_{0}a\otimes\xi) & =W(R(b\otimes I_{F})R^{*})P_{0}a\otimes\xi\\
= & (b\otimes I_{F})R^{*}a\xi\\
= & (b\otimes I_{F})W(P_{0}a\otimes\xi).\end{align*}
The fact that $W$ is isometric is another easy computation: For all
$a,b\in M$, and $\xi,\eta\in F$, \begin{align*}
\langle P_{0}a\otimes\xi,P_{0}b\otimes\eta\rangle & =\langle\xi,a^{*}P_{0}b\eta\rangle\\
= & \langle P_{0}a\xi,P_{0}b\eta\rangle\\
= & \langle R^{*}a\xi,R^{*}b\eta\rangle\\
= & \langle W(P_{0}a\otimes\xi),W(P_{0}b\otimes\eta)\rangle.\end{align*}
(Note that we have used the fact that $P_{0}=RR^{*}$ when passing
from the second line to the third.) Since $\mathcal{K}(\mathcal{F}(F))\mathcal{F}(F)=\mathcal{F}(F)$,
$P_{0}\mathcal{K}(\mathcal{F}(F))\mathcal{F}(F)=P_{0}\mathcal{F}(F)$,
and so $R^{*}P_{0}\mathcal{K}(\mathcal{F}(F))\mathcal{F}(F)=R^{*}P_{0}\mathcal{F}(F)=\mathcal{F}(F)\otimes F$.
This shows that $W$ is surjective.\end{proof}
\begin{defn}
Given a $C^{*}$-correspondence pair $(F,N)$, we call the $C^{*}$-correspondence
pair $(E,M)=(P_{0}\mathcal{K}(\mathcal{F}(F)),\mathcal{K}(\mathcal{F}(F))$
constructed in Proposition \ref{pro:Stable_Morita_equivalence} the
\emph{canonical stabilization} of $(F,N)$, and we call $(\mathcal{F}(F),W)$
the \emph{canonical (E,M)$,$$(F,N)$-equivalence}. 
\end{defn}
We want to illustrate the calculations of Proposition \ref{pro:Stable_Morita_equivalence}
in a concrete setting first considered by Popescu. For this purpose,
we require two observations. First, recall that $E$ has the form
$PM$. In general, if $M$ is a $C^{*}$-algebra and if $E$ has the
form $PM$, where $P$ is a projection in the multiplier algebra of
$M$, then we called $(E,M)$ \emph{strictly cyclic} in \cite[Page 419]{MS98b}.
In this case, if $(T,\sigma)$ is a completely contractive covariant
representation of $(E,M)$ on a Hilbert space $H$, then $E\otimes_{\sigma}H$
is really $\sigma(P)H$, where we have extended $\sigma$ to the multiplier
algebra of $M$, if $M$ is not unital. Consequently, the intertwiner
$\widetilde{T}$ really maps the \emph{subspace} $\sigma(P)H$ into
$H$ but the adjoint of $\widetilde{T}$ may be viewed as an operator
\emph{on} $H$, i.e., from $H$ to $H$, with range contained in $\sigma(P)H$,
of course. Second, observe that in general, if $(T,\sigma)$ is a
covariant representation of $(F,N)$ on a Hilbert space $H$, then the
representation induced from the canonical equivalence is $(T^{\mathcal{F}(F)},\sigma^{\mathcal{F}(F)})$.
We know $\sigma^{\mathcal{F}(F)}$ represents $\mathcal{K}(\mathcal{F}(F))$
on $\mathcal{F}(F)\otimes_{\sigma}H$ via the ordinary action of $\mathcal{K}(\mathcal{F}(F))$
on $\mathcal{F}(F)$, tensored with the identity operator on $H$,
i.e., $\sigma^{\mathcal{F}(H)}(a)=a\otimes I_{H}$. On the other hand,
from Theorem \ref{thm:functor}, $\widetilde{T^{\mathcal{F}(F)}}=(I_{\mathcal{F}(F)}\otimes\widetilde{T})(W\otimes I_{H})$.
But as we noted in the proof of Proposition \ref{pro:Stable_Morita_equivalence},
$W$ is effectively $R^{*}$, and taking into account all the balancing
that is taking place, we may write $\widetilde{T^{\mathcal{F}(F)}}=(I_{\mathcal{F}(F)}\otimes\widetilde{T})(R^{*}\otimes I_{H})$.
Since, as we just remarked, $\widetilde{T^{\mathcal{F}(F)}}$ maps
from $E\otimes_{M}\mathcal{F}(F)\otimes_{\sigma}H=P_{0}\mathcal{K}(\mathcal{F}(F))\otimes_{\mathcal{K}(\mathcal{F}(F))}\mathcal{F}(F)\otimes_{\sigma}H$,
which can be identified with the subspace $P_{0}\mathcal{F}(F)\otimes_{\sigma}H$
of $\mathcal{F}(F)\otimes_{\sigma}H$, it will be more convenient
in the example below to work with the adjoint of $\widetilde{T^{\mathcal{F}(F)}}$,\begin{equation}
\left(\widetilde{T^{\mathcal{F}(F)}}\right)^{*}=(R\otimes I_{H})(I_{\mathcal{F}(F)}\otimes\widetilde{T}^{*}),\label{eq:T-star}\end{equation}
and view $\left(\widetilde{T^{\mathcal{F}(F)}}\right)^{*}$ as an
operator in $B(\mathcal{F}(F)\otimes_{\sigma}H)$.
\begin{example}
\label{exa:Popescu's_reconstruction_operator}In this example, we
let $N=\mathbb{C}$ and we let $F=\mathbb{C}^{d}$. We interpret $\mathbb{C}^{d}$
as $\ell^{2}(\mathbb{N})$, if $d=\infty$. If $(T,\sigma)$ is a
completely contractive covariant representation of $(\mathbb{C}^{d},\mathbb{C})$
on a Hilbert space $H$, then $\sigma$ is just the $n$-fold multiple
of the identity representation of $\mathbb{C}$, where $n$ is the
dimension of $H$. Also, $\widetilde{T}$ may be viewed in terms of
a $1\times d$ matrix of operators on $H$, $[T_{1},T_{2},\cdots,T_{d}]$,
such that $\sum_{i=1}^{d}T_{i}T_{i}^{*}\leq I_{H}$, i.e. $[T_{1},T_{2},\cdots,T_{d}]$
is a row contraction. When $\mathbb{C}^{d}\otimes H$ is identified
with the column direct sum of $d$ copies of $H$, the formula for
$\widetilde{T}:\mathbb{C}^{d}\otimes H\to H$ is $\widetilde{T}(\left(\begin{array}{c}
h_{1}\\
h_{2}\\
\vdots\\
h_{d}\end{array}\right))=\sum_{i=1}^{d}T_{i}h_{i}$. Consequently, $\widetilde{T}^{*}:H\to\mathbb{C}^{d}\otimes H$ is
given by the formula \[
\widetilde{T}^{*}h=\left(\begin{array}{c}
T_{1}^{*}h\\
T_{2}^{*}h\\
\vdots\\
T_{d}^{*}h\end{array}\right).\]
On the other hand, $\mathcal{F}(\mathbb{C}^{d})\otimes\mathbb{C}^{d}$
may be viewed as the column direct sum of $d$ copies of $\mathcal{F}(\mathbb{C}^{d})$
and when this is done, $R$ has a matricial representation as $[R_{1},R_{2},\cdots,R_{d}]$,
where $R_{i}$ is the right creation operator on $\mathcal{F}(\mathbb{C}^{d})$
determined by the $i^{th}$ canonical basis vector $e_{i}=(0,0,\cdots,0,1,0,\cdots,0)^{\intercal}$
for $\mathbb{C}^{d}$, i.e., $R_{i}\xi=\xi\otimes e_{i}$. Notice
that $[R_{1},R_{2},\cdots,R_{d}]$ is a \emph{row isometry,} meaning
that $R_{i}$'s are all isometries and that their range projections
$R_{i}R_{i}^{*}$ are mutually orthogonal. Thus, the formula for $R:\mathcal{F}(\mathbb{C}^{d})\otimes\mbox{C}^{d}\to\mathcal{F}(\mathbb{C}^{d})$
is $R\left(\begin{array}{c}
\xi_{1}\\
\xi_{2}\\
\vdots\\
\xi_{d}\end{array}\right)=\sum_{i=1}^{d}R_{i}\xi_{i}$. Consequently, in the context of this example, equation \eqref{eq:T-star}
becomes \begin{align*}
\left(\widetilde{T^{\mathcal{F}(\mathbb{C}^{d})}}\right)^{*}(\xi\otimes h) & =(R\otimes I_{H})(I_{\mathcal{F}(\mathbb{C}^{d})}\otimes\widetilde{T}^{*})(\xi\otimes h)\\
= & (R\otimes I_{H})\left(\begin{array}{c}
\xi\otimes T_{1}^{*}h\\
\xi\otimes T_{2}^{*}h\\
\vdots\\
\xi\otimes T_{d}^{*}h\end{array}\right)\\
= & \sum_{i=1}^{d}R_{i}\xi\otimes T_{i}^{*}h\\
= & (\sum_{i=1}^{d}R_{i}\otimes T_{i}^{*})(\xi\otimes h),\end{align*}
i.e., $\left(\widetilde{T^{\mathcal{F}(\mathbb{C}^{d})}}\right)^{*}$
is Popescu's \emph{reconstruction operator} $\sum_{i=1}^{d}R_{i}\otimes T_{i}^{*}$. 
\end{example}
The reconstruction operator first appeared implicitly in \cite{gP89},
where Popescu developed a characteristic operator function for noncommuting
$d$-tuples of contractions. (In this connection it was used explicitly
in \cite{gP2006}.) The first place the term ``reconstruction operator''
appeared in the literature is \cite[Page 50]{gP2009}, which began
circulating as a preprint in 2004. Since that time, the reconstruction
operator has played an increasingly prominent role in Popescu's work. In
addition, the reconstruction operator has popped up elsewhere in the
literature, but without the name attached to it.  One notable example is Orr
Shalit's paper \cite[Page 69]{oS2008}. There he attached a whole semigroup of
them to representations of certain product systems of correspondences. 
Because of Example \ref{exa:Popescu's_reconstruction_operator} we
feel justified in introducing the following terminology.
\begin{defn}
\label{def:Reconstruction_operator}If $(T,\sigma)$ is a completely
contractive covariant representation of a $C^{*}$-correspondence
pair $(F,N)$ on a Hilbert space $H$, then the adjoint of the intertwiner
of the Morita transform of the canonical stabilization of $(F,N)$
is called the \emph{reconstruction operator} of $(T,\sigma)$; i.e.,
the reconstruction operator of $(T,\sigma)$ is defined to be $(\widetilde{T^{\mathcal{F}(F)}})^{*}$
viewed as an operator in $B(\mathcal{F}(F)\otimes_{\sigma}H)$. 
\end{defn}
Our analysis begs the questions: How unique is the canonical stabilization
of a $C^{*}$-correspondence pair? Are there non-canonical stabilizations?
In general there are many stabilizations that ``compete'' with the
canonical stabilization. Organizing them seems to be a complicated
matter. To see a little of what is possible, we will briefly outline
what happens in the setting of Example \ref{exa:Popescu's_reconstruction_operator}.
So fix $(\mathbb{C}^{d},\mathbb{C})$. We shall assume $d$ is finite
to keep matters simple. We can stabilize $\mathbb{C}$ as a $C^{*}$-
algebra getting the compact operators on $\ell^{2}(\mathbb{N})$.
It is important to do this explicitly, however. So let $\mathcal{X}$
be column Hilbert space $\mathbf{C}_{\infty}$. This is $\ell^{2}(\mathbb{N})$
with the operator space structure it inherits as the set of all operators
from $\mathbb{C}$ to $\ell^{2}(\mathbb{N})$. Equivalently, it is
the set of all infinite matrices $T=(t_{ij})$ that represent a compact
operator on $\ell^{2}(\mathbb{N})$ and have the property that $t_{ij}=0$,
when $j>1$. (See \cite{BMP2000}.) We then have $\widetilde{\mathbf{C}_{\infty}}=\mathbf{R}_{\infty}$,
the row Hilbert space. Also, if $\mathcal{K}=\mathcal{K}(\ell^{2}(\mathbb{N}))$,
then $\mathbf{C}_{\infty}$ is a $\mathcal{K},\mathbb{C}$-equivalence
bimodule. So, if $E$ is any correspondence over $\mathcal{K}$ that
is equivalent to $\mathbb{C}^{d}$, then $E$ must be isomorphic to
\[
\mathbf{C}_{\infty}\otimes_{\mathbb{C}}\mathbb{C}^{d}\otimes_{\mathbb{C}}\mathbf{R}_{\infty}\simeq\mathbf{C}_{d}(\mathcal{K})\]
with its usual left and right actions of $\mathcal{K}$. Because $\mathcal{K}$
is stable, there is an endomorphism $\alpha$ of $\mathcal{K}$ such
that $\mathbf{C}_{d}(\mathcal{K})$ is isomorphic to $_{\alpha}\mathcal{K}$.
That is, $_{\alpha}\mathcal{K}$ is $\mathcal{K}$ as a right $\mathcal{K}$-module
(the module product is just the product in $\mathcal{K}$ and the
$\mathcal{K}$-valued inner product is $\langle\xi,\eta\rangle:=\xi^{*}\eta$.)
The left action of $\mathcal{K}$ is that which is implemented by
$\alpha$, i.e., $a\cdot\xi:=\alpha(a)\xi$. General theory tells
us this is the case, but we can see it explicitly as follows. Choose
a Cuntz family of $d$ isometries on $\ell^{2}(\mathbb{N})$, $\{S_{i}\}_{i=1}^{d}$.
(This means that $S_{i}^{*}S_{j}=\delta_{ij}I$ and $\sum_{i=1}^{d}S_{i}S_{i}^{*}=I$.)
Then, as is well known, $\{S_{i}\}_{i=1}^{d}$ defines an endomorphism
of $\mathcal{K}$ via the formula $\alpha(a)=\sum_{i=1}^{d}S_{i}aS_{i}^{*}$.
Note, too, that $\alpha$ extends to be a \emph{unital} endomorphism
of $B(\ell^{2}(\mathbb{N}))$ since $\sum_{i=1}^{d}S_{i}S_{i}^{*}=I$.
On the other hand, define $V:\mathbf{C}_{d}(\mathcal{K})\to\mathcal{K}$
via the formula \[
V(\left(\begin{array}{c}
a_{1}\\
a_{2}\\
\vdots\\
a_{d}\end{array}\right))=\sum_{i=1}^{d}S_{i}a_{i},\qquad\left(\begin{array}{c}
a_{1}\\
a_{2}\\
\vdots\\
a_{d}\end{array}\right)\in\mathbf{C}_{d}(\mathcal{K}).\]
Then it is a straightforward calculation to see that $V$ is a correspondence
isomorphism from $\mathbf{C}_{d}(\mathcal{K})$ onto $_{\alpha}\mathcal{K}$.
Thus $\mathcal{X}=\mathbf{C}_{\infty}$ is an equivalence bimodule
between $(_{\alpha}\mathcal{K},\mathcal{K})$ and $(\mathbb{C}^{d},\mathbb{C})$
and $(_{\alpha}\mathcal{K},\mathcal{K})$ is a bona fide contender
for a stabilization of $(\mathbb{C}^{d},\mathbb{C})$. Note that this
time $_{\alpha}\mathcal{K}$ is strictly cyclic, but the projection
$P$ is the identity.

Suppose, now, that $(T,\sigma)$ is a completely contractive covariant
representation of $(\mathbb{C}^{d},\mathbb{C})$ on a Hilbert space
$H$. Then as before $\sigma$ is an $n$-fold multiple of the identity
representation of $\mathbb{C}$ on $\mathbb{C}$, where $n$ is the
dimension of $H$ and $\widetilde{T}:\mathbb{C}^{d}\otimes H\to H$
may be viewed as a row contraction $[T_{1},T_{2},\cdots,T_{d}]$ of
operators on $H$. The induced representation of $\mathcal{K}$, $\sigma^{\mathbf{C}_{\infty}}$
is the $n$-fold multiple of the identity representation of $\mathcal{K}$
(same $n$) and a calculation along the lines of that was carried
out in Example \ref{exa:Popescu's_reconstruction_operator} shows
that $\left(\widetilde{T^{\mathbf{C}_{\infty}}}\right)^{*}=S_{1}\otimes T_{1}^{*}+S_{2}\otimes T_{2}^{*}+\cdots+S_{d}\otimes T_{d}^{*}$
acting on $\ell^{2}(\mathbb{N})\otimes H$. Thus, $\left(\widetilde{T^{\mathbf{C}_{\infty}}}\right)^{*}$
is an alternative for Popescu's reconstruction operator. How different
from his reconstruction operator $\left(\widetilde{T^{\mathbf{C}_{\infty}}}\right)^{*}$
is remains to be seen. We believe the difference could be very interesting.
We believe that the dependence of $\left(\widetilde{T^{\mathbf{C}_{\infty}}}\right)^{*}$
on the Cuntz family $\{S_{i}\}_{i=1}^{d}$ could be very interesting,
also.\bigskip

\noindent\emph{Acknowledgment:} We are very grateful to Gelu Popescu
for giving us some background and references on his reconstruction
operator.

\end{document}